\documentclass[preprint,12pt]{elsarticle}
\usepackage{tikz}
\usepackage{amsthm,amsfonts,amssymb,amscd,amsmath,enumerate,verbatim,calc,graphicx,geometry}
\usepackage[all]{xy}
\newtheorem{theorem}{Theorem}[section]

\newtheorem{corollary}[theorem]{Corollary}
\theoremstyle{definition}
\theoremstyle{definitions}
\newtheorem{definition}[theorem]{Definition}

\newtheorem{remark}[theorem]{Remark}

\theoremstyle{notations}

\theoremstyle{remarks}

\journal{}

\begin{document}

\begin{frontmatter}



\title{Comparison of Topologies on Homotopy Groups with Subgroup Topology Viewpoint }


\author[N. Shahami]{Naghme Shahami }
\ead{na.shahami@mail.um.ac.ir}
\author[B. Mashayekhy]{Behrooz Mashayekhy\corref{cor1}}
\ead{bmashf@um.ac.ir }

\address{Department of Pure Mathematics, Faculty of Mathematical Sciences, Ferdowsi University of Mashhad,
P.O.Box 1159, Mashhad 91775, Iran.}
\cortext[cor1]{Corresponding author}

\begin{abstract}
By introducing various topologies on the homotopy groups of a topological space, 
some researchers make these well known notions in algebraic topology more useful 
and powerful. In this paper, first we recall and review some known topologies on homotopy groups. 
Then by reviewing some famous subgroups of homotopy groups 
and using the concept of subgroup topology, we intend to compare these topologies in order to present some results on topologized homotopy groups.
\end{abstract}


\end{frontmatter}

\section{Introduction and Motivation}

During last two decades some researchers have shown that there are various useful, interesting and functorial topologies 
on the fundamental group $\pi_1(X,x_0)$  
and the homotopy groups $\pi_n(X,x_0)$ of a topological space $X$ which make them powerful tools for studying some 
topological properties of spaces 
(see \cite{A16,A20,BKP21,BKP24,B13,BF23,CM,F9,F12,G8,SM,W}).  Recently, the authors have studied and reviewed most of known topologies on 
the fundamental group $\pi_1(X,x_0)$ and summed up all comparison of various topologies on the fundamental group in a diagram 
(see \cite{SM}). 
In this paper, we are going to study and review some topologies on the homotopy groups $\pi_n(X,x_0)$ 
and then using the concept of subgroup topology 
with respect to some neighbourhood families of famous subgroups of the homotopy groups $\pi_n(X,x_0)$, 
we intend to compare topologized homotopy groups 
with various topologies.

The paper is organized as follows. Inspired by some famous subgroups of the fundamental group $\pi_1(X,x_0)$, some researchers have introduced 
and defined similar subgroups for the homotopy group $\pi_n(X,x_0)$ (see \cite{BKP21,BKP24}). 
In Section 2, we review these subgroups including the $n$-small subgroup,
$\pi_n^{s}(X,x_0)$, the $n$-small generated subgroup, $\pi_n^{sg}(X,x_0)$, the $n$-Spanier subgroup, 
$\pi_n^{sp}(X,x_0)$, the $n$-thick Spanier subgroup, 
$\pi_n^{tsp}(X,x_0)$, the $n$-weak thick Spanier subgroup, $\pi_n^{wtsp}(X,x_0)$. 
Also, we are inspired by \cite{TMP} to define the n-path Spanier subgroup, 
$\widetilde{\pi}_n^{sp}(X,x_0)$, and compare this subgroup to the others as follows 
(in order to avoid confusion, from now on, we omit the subscript $n$ for the subgroups of the homotopy group $\pi_n(X,x_0)$):
\[\pi^s(X,x_0) \leq \pi^{sg}(X,x_0) \leq \widetilde{\pi}^{sp}(X,x_0) \leq \pi^{sp}(X,x_0) 
\leq \pi^{tsp}(X,x_0) \leq \pi^{wtsp}(X,x_0). \]

In Section 3, first we review various topologies that have been defined on homotopy groups. 
These topologies include the whisker topology \cite{BM20}, $\pi_n^{wh}(X,x_0)$, 
the  compact-open quotient topology \cite{G8, G10}, $\pi_n^{qtop}(X,x_0)$, the tau topology \cite{B13}, $\pi_n^{\tau}(X,x_0)$, 
the lim topology \cite{G10}, $\pi_n^{lim}(X,x_0)$, 
the Spanier topology \cite{BKP24}, $\pi_n^{Span}(X,x_0)$, the thick Spanier topology \cite{BKP24}, $\pi_n^{tSpan}(X,x_0)$, 
the weak thick Spanier topology \cite{BKP24}, $\pi_n^{wtSpan}(X,x_0)$, 
the shape topology \cite{BF23}, $\pi_n^{sh}(X,x_0)$, and the pseudometric topology \cite{BF23}, $\pi_n^{met}(X,x_0)$. 
Then by reviewing the concept of subgroup topology of a 
neighbourhood family introduced in \cite{BS} especially a type of subgroup topology with respect a subgroup introduced in \cite{SM}
and considering some famous subgroups of the homotopy groups mentioned in Section 2, 
we introduce some topologies on the homotopy groups $\pi_n(X,x_0)$ such as
the path Spanier topology, $\pi_n^{pSpan}(X,x_0)$, the $n$-small subgroup topology, $\pi_n^{\pi^{s}(X,x_0)}(X,x_0)$, 
the $n$-small generated subgroup topology, $\pi_n^{\pi^{sg}(X,x_0)}(X,x_0)$, the $n$-path Spanier subgroup topology, 
$\pi_n^{\widetilde{\pi}^{s}(X,x_0)}(X,x_0)$, the $n$-Spanier subgroup topology, $\pi_n^{\pi^{sp}(X,x_0)}(X,x_0)$, 
the $n$-thick Spanier subgroup topology, $\pi_n^{\pi^{tsp}(X,x_0)}(X,x_0)$, the $n$-weak thick Spanier subgroup topology, 
$\pi_n^{\pi^{wtsp}(X,x_0)}(X,x_0)$.

In Section 4, using some properties of the subgroup topology presented in \cite{BS, SM} and some results of other researchers, 
we are going to compare the above mentioned topologies on homotopy groups with each other as much as possible.
Among some results, we prove that 
\[\pi_n^{sh}(X,x_0) \preccurlyeq \pi_n^{wtSpan}(X,x_0) \preccurlyeq \pi_n^{tSpan}(X,x_0) 
\preccurlyeq \pi_n^{Span}(X,x_0) \preccurlyeq \pi_n^{pSpan}(X,x_0),\]
and
\[\pi_n^{\tau}(X,x_0) \preccurlyeq \pi_n^{qtop}(X,x_0) \preccurlyeq \pi_n^{lim}(X,x_0)=\pi_n^{wh}(X,x_0) 
\preccurlyeq \pi_n^{\pi^{s}(X,x_0)}(X,x_0),\]
where $A \preccurlyeq B$ mean that $B$ is finer than $A$.
Also we show that
\[ \pi_n^{\pi^{wtsp}(X,x_0)}(X,x_0) \preccurlyeq \pi_n^{\pi^{tsp}(X,x_0)}(X,x_0) \preccurlyeq \pi_n^{\pi^{sp}(X,x_0)}(X,x_0) 
\preccurlyeq \pi_n^{\widetilde{\pi}^{s}(X,x_0)}(X,x_0)\] 
\[\preccurlyeq \pi_n^{\pi^{sg}(X,x_0)}(X,x_0) \preccurlyeq  \pi_n^{\pi^{s}(X,x_0)}(X,x_0).\]
Moreover, we investigate some conditions under which some of the above inequalities become equality.
As an example, we define the concept $n$-semilocally $H$-connected at $x_0$ for a subgroup $H \leq \pi_n(X,x_0)$ and 
we prove that a topological space $X$ is $n$-semilocally $H$-connected at $x_0$ if and only if $H$ is an open subgroup of $\pi_n^{wh}(X,x_0)$
which is a generalization of \cite[Theorem 4.2]{A16}. We show that the equality $\pi_n^{wh}(X,x_0)= \pi_n^{\pi^s(X,x_0)}(X,x_0)$ holds 
if and only if $X$ is $n$-semilocally $H$-connected at $x_0$.
Finally, we sum up all results of this section in a diagram in order to compare various topologies on homotopy groups.

\section{Some Subgroups of Homotopy Groups}

In order to review some subgroups on the homotopy groups similar to famous subgroups of the fundamental group, 
we recall some notions and notations from \cite{Naber}.
A loop at $x_0$ in $X$ is a map from $I = [0, 1]$ into $X$ that carries the boundary $\partial I = \{0, 1\}$ onto $x_0$. 
The fundamental group $\pi_1(X, x_0)$ 
consists of all homotopy classes, relative to $\partial I$, of such loops and admits a natural, and very useful, group structure. 
This concept is generalized as follows. 
For each positive integer $n$, $I^n = [0, 1]^n = \{(s_1,...,s_n) \in \mathbb{R}^n : 0 \leq s_i \leq 1, i = 1,...,n\}$ 
is called the n-cube in $\mathbb{R}^n$. 
The boundary $\partial I^n$ of $I^n$ consists of all $(s_1,...,s_n) \in I^n$ for which $s_i = 0$ or $s_i = 1$ for at least one value of $i$. 
If $X$ is a topological space 
and $x_0 \in X$, then an n-loop at $x_0$ is a continuous map $\alpha : I^n \to X$ such that $\alpha (\partial I^n) = {x_0}$. 
The collection of all n-loops at $x_0$ in $X$ 
is partitioned into equivalence classes by homotopy relative to $\partial I^n$. The equivalence class containing an n-loop $\alpha$ 
is denoted by $[\alpha]$. If $\alpha$ is 
an n-loop at $x_0$ in $X$, define $\alpha^{-1} : I^n \to X$ by $\alpha^{-1}(s_1, s_2,...,s_n) = \alpha(1 -  s_1, s_2,...,s_n)$ for all 
$(s_1, s_2,...,s_n) \in I^n$ (see \cite[p. 136]{Naber}). Let $X$ be a topological space, $x_0 \in X$ and $n \geq 2$ a positive integer. 
Let $\pi_n(X, x_0)$ 
be the set of all homotopy classes, relative to $\partial I^n$, of n-loops at $x_0$ in $X$. For $[\alpha], [\beta] \in \pi_n(X, x_0)$, define 
$[\alpha] + [\beta]=[\alpha + \beta]$. Then, with this operation, $\pi_n(X, x_0)$ is an abelian group in which the identity element is $[c_{x_0}]$ 
and the inverse of any $[\alpha]$ is given by $[\alpha^{-1}]$ (see \cite[Theorem 2.5.3]{Naber}). The group $\pi_n(X,x_0)$ is called the 
\emph{n-th homotopy group} of $X$ at $x_0$.

We know that if there exists a path $\sigma : I \to X$ from $x_0$ to $x_1$, then $\pi_1(X,x_0)$ and $\pi_1(X,x_1)$ are isomorphism. 
A similar result exists for homotopy groups. Let $n \geq 2$, define $\sigma_{\#}: \pi_n(X,x_1) \to \pi_n(X,x_0)$ by $\sigma_{\#}[\alpha] = [F_1]$, 
where $F_1(s) = F(s,1)$ for any $s \in I^n$ and $F : I^n \times I \to X$ is a homotopy with the following properties:
\[ F(s,0) = \alpha (s) ; s \in I^n\]
\[ F(s,t) = \sigma^{-1}(t) ; s \in \partial I^n , t \in I.\]
Then by \cite[Theorem 2.5.6]{Naber} $\sigma_{\#}$ is a well-defined isomorphism that depends only on the path homotopy class of $\sigma$, i.e., 
if $\sigma^{\prime} \simeq \sigma$ rel $\{0, 1\}$, then $\sigma^{\prime}_{\#} = \sigma_{\#}$.\\

In \cite{BKP21,BKP24}, Bahredar et al. defined some subgroups of homotopy groups similar to some famous subgroups of 
the fundamental groups as follows.

1. \textbf{The n-small subgroup}: An n-loop $\alpha : (I^n, \partial I^n) \to (X,x_0)$ 
is said to be small if for every open neighborhood $U$ of $x_0$, 
there exists an n-loop $f : I^n \to X$ with the base point $x_0$ such that $f(I^n ) \subseteq U$ and $[\alpha] = [f]$ \cite[Definition 4.5]{BKP21}. 
A \textit{small n-loop subgroup at $x_0$}, denoted by $\pi^s(X,x_0)$, is a subgroup of $\pi_n(X, x_0)$ 
consisting of homotopy classes of all small n-loops 
at $x_0$ (see \cite{BKP21}).

2. \textbf{The n-small generated subgroup}: The set 
$\pi^{sg}(X,x_0) = \{\sigma_{\#}([\alpha]) | [\alpha] \in \pi^s(X,x_0),$ $\sigma \text{ is a path with the end point } x_0\}$ is a subgroup of 
$\pi_n(X,x_0)$ which is called the \textit{n-small generated subgroup} (see \cite{BKP21}).

3. \textbf{The n-Spanier subgroup}: Let $\mathcal{U}$ be an open cover of a pointed space $(X,x_0)$. 
Then $\pi^{sp}(\mathcal{U},x_0)$ is a subgroup of 
$\pi_n(X,x_0)$ which is spanned by all homotopy classes of the form $\prod_{j=1}^{n}\sigma_{{j}_{\#}}[v_j]$, 
where for every $1 \leq j \leq n, \sigma_j(0) = x_0$ 
and the n-loop $v_j$ lies in one of the neighborhoods $U_j \in \mathcal{U}$. This group is called the n-Spanier subgroup 
with respect to $\mathcal{U}$
 (see \cite[Definition 3.1]{BKP21}).  The intersection of the n-Spanier subgroups relative to open covers of $X$ 
 is called the \textit{n-Spanier subgroup} 
 and it is denoted by $\pi^{sp}(X,x_0)$ (see \cite{BKP21}).

4. \textbf{The n-thick Spanier subgroup}: Let $\mathcal{U}$ be an open cover of a pointed space $(X,x_0)$. 
The n-thick Spanier subgroup with respect to 
$\mathcal{U}$ is the subgroup of $\pi_n(X,x_0)$ which is generated by elements of the form $\sigma_{\#}[f_1 \circledast f_2]$, 
where $f_1 \circledast f_2 (z) = f_1(z)$ 
if $z \in S^n_+$ and $f_1 \circledast f_2 (z) = f_2(z)$ if $z \in S^n_-$, $f_1, f_2: (D^n,s_0) \to (X,x_0)$ 
are pointed continuous maps such that for any 
$z \in \partial D^n, f_1(z) = f_2(z)$, and $Im f_i \subseteq U_i$ for some $U_i \in \mathcal{U}$ $(i = 1,2)$ 
and it is denoted by $\pi^{tSp}(\mathcal{U},x_0)$. 
The intersection of the n-thick Spanier subgroups relative to open covers of $X$ is called the \textit{n-thick Spanier subgroup} 
which is denoted by $\pi^{tsp}(X,x_0)$ 
(see \cite{BKP24}).

5. \textbf{The n-weak thick Spanier subgroup}: Let $\mathcal{U}$ be an open cover of a pointed space $(X,x_0)$. 
The n-weak thick Spanier subgroup with respect to $\mathcal{U}$ is the subgroup of $\pi_n(X,x_0)$ 
which is generated by elements of the form $\sigma_{\#}[\alpha]$, 
where $\alpha (I^n) \subseteq U_1 \cup U_2$ for some $U_1,U_2 \in \mathcal{U}$ and it is denoted by $\pi^{wtsp}(\mathcal{U},x_0)$. 
The \textit{n-weak thick Spanier subgroup} of $X$ is the intersection of the $n$-weak thick Spanier subgroups relative to open covers of $X$
 which is denoted by $\pi^{wtsp}(X,x_0)$ (see \cite{BKP24}).

6. \textbf{The n-path Spanier subgroup}: We are inspired by \cite{TMP} to define the n-path Spanier subgroup. 
A path open cover of a pointed space $(X,x_0)$ 
is an open cover $\mathcal{V} = \{V_{\alpha} | \alpha \in P(X,x_0)\}$  such that $\alpha (1) \in V_{\alpha}$. 
Then we define $\widetilde{\pi}^{sp}(\mathcal{V},x_0)$ 
as a subgroup of $\pi_n(X,x_0)$ which is spanned by all homotopy classes of the form $\prod_{j=1}^{n}\alpha_{{j}_{\#}}[v_j]$, where for every 
$1 \leq j \leq n, \alpha_j(0) = x_0$ and the n-loop $v_j$ lies in $V_{\alpha_j} \in \mathcal{V}$. 
We called this subgroup the $n$-path Spanier subgroup 
with respect to the path open cover $\mathcal{V}$.  The intersection of $n$-path Spanier subgroups relative to path open covers of $X$ is called the 
\textit{n-path Spanier subgroup} and we denote it by $\widetilde{\pi}^{sp}(X,x_0)$.

\begin{remark}
(i) By definition of n-Spanier subgroups and n-path Spanier subgroups, it easy to see that $ \widetilde{\pi}^{sp}(X,x_0) \leq \pi^{sp}(X,x_0)$.
 Moreover, let $\mathcal{V} = \{V_{\beta} | \beta \in P(X,x_0)\}$ be a path open cover of $X$ and let $x_0 \in V_{\beta_0}$. 
 If $\alpha$ be a small n-loop 
 at $x_0$ in $X$, there exists an n-loop $f : I^n \to X$ such that $f (I^n) \subseteq V_{\beta_0}$ and $[f] = [\alpha]$, thus 
 $\sigma_{\#}([\alpha]) = \sigma_{\#}([f])$, where $\sigma_{\#} ([\alpha]) \in \pi^{sg}(X,x_0)$ and 
 $\sigma_{\#}([f]) \in \widetilde{\pi}^{sp}(\mathcal{V},x_0)$. 
 Hence $\pi^{sg}(X,x_0) \subseteq \widetilde{\pi}^{sp}(\mathcal{V},x_0)$ for every path open cover $\mathcal{V}$. Therefore, we have
 \[\pi^{sg}(X,x_0) \leq \widetilde{\pi}^{sp}(X,x_0) \leq \pi^{sp}(X,x_0).\]
(ii) By \cite[Proposition 4.6]{BKP21}, \cite[Proposition 3.15]{BKP24} and the above note we have the following chain of subgroups of $\pi_n(X,x_0)$.
\[\pi^s(X,x_0) \leq \pi^{sg}(X,x_0) \leq \widetilde{\pi}^{sp}(X,x_0) \leq \pi^{sp}(X,x_0)\]
 \[ \leq \pi^{tsp}(X,x_0) \leq \pi^{wtsp}(X,x_0) \leq ker\Psi_X\]
where $\Psi_X$ is the canonical map from the $n$-th homotopy group to the $n$-th shape homotopy groups (see \cite{BF23}).
\end{remark}

\section{Some Topologies on Homotopy Groups}

In order to review and introduce some topologies on the homotopy groups similar to famous topologies of the fundamental group, 
we recall the notion of subgroup topology 
of a neighbourhood family introduced in \cite{BS} especially a type of subgroup topology with respect to a subgroup introduced in \cite{SM}.

A collection $\Sigma$ of subgroups of $G$ is called a \textit{neighbourhood family} if for any $H,K \in \Sigma$, there is a subgroup $S \in \Sigma$ 
such that $S \subseteq H\cap K$. As a result of this property, the collection of all left cosets of elements of $\Sigma$ 
forms a basis for a topology on $G$, 
which is called the subgroup topology determined by $\Sigma$ and we denote it by $G^{\Sigma}$. 
The \textit{subgroup topology} on a group $G$ specified 
by a neighbourhood family was defined in \cite[Section 2.5]{BS} and considered by some recent researchers such as \cite{A20,W}.
Since left translation by elements of $G$ determine self-homeomorphisms of $G$, they are homogeneous spaces. 
 Bogley et al. \cite{BS} focused on some general properties of subgroup topologies and by introducing the intersection 
 $S_\Sigma=\cap \{H \ \vert \ H \in \Sigma\}$, 
 called the \textit{infinitesimal} subgroup for the neighbourhood family $\Sigma$. They showed that the closure of the element $g\in G$ 
 is the coset $g S_\Sigma$. 

Let H be a subgroup of a group $G$. Then we define $\Sigma^H$ as follows
\[\Sigma^H= \{K\leqslant G \ | \ H\subseteq K\}.\]
It is easy to see that $\Sigma^H$ is a neighbourhood family. We consider the subgroup topology on $G$ determined by $\Sigma^H$ 
and denote it by $G^H$. 
Note that the infinitesimal subgroup for the neighbourhood family $\Sigma^H$ is $H$.  We have compared two subgroup topologies on a group. 
If $\Sigma$ and $\Sigma^{\prime}$ are two neighbourhood families on $G$ such that $G^{\Sigma} = G^{\Sigma^{\prime}}$, 
then $S_{\Sigma} = S_{\Sigma^{\prime}}$ \cite[Theorem 2.3]{SM}. Let $\Sigma$ and $\Sigma^{\prime}$ 
be two neighbourhood families on $G$ such that $S_{\Sigma} \leq S_{\Sigma^{\prime}}$ and $S_{\Sigma} \in \Sigma$ then
 $G^{\Sigma^{\prime}} \preccurlyeq G^{\Sigma}$. In particular, if $H\leq K\leq G$, then $G^K \preccurlyeq G^H$.
 Furthermore, in \cite{SM} we proved that  $G^H$ is discrete if and only if $H=1$. Also, $G^H$ is indiscrete if and only if $H=G$.  
 It is pointed out in \cite{BS} that although the group $ G $ equipped with a subgroup topology may not necessarily be a topological group, 
 in general (it may not even be a quasitopological group), because right translation maps by a fixed element of $G$ need not be continuous, 
 but it has some of properties of topological groups  (for more details see Theorem 2.9 in \cite{BS}). Wilkins \cite[Lemma 5.4]{W} 
 showed that a group $G$ 
 with the subgroup topology determined by a neighbourhood family $\Sigma$ is a topological group when all subgroups in $\Sigma$ are normal.
 Moreover, it is proved in \cite[Corollary 2.2]{A20} that a group equipped with a subgroup topology is a topological group 
 if and only if all right translation maps are continuous.
Using these facts we can easily see that $G^H$ is a topological group if $H$ is a normal subgroup of $G$ and $G/H$ is an abelian group. 
We vastly extended these results in \cite[Theorem 2.2]{SM}, let $G$ be a group and $\Sigma$ be a neighbourhood family on $G$ 
such that $S_{\Sigma}\in \Sigma $, 
then $G^{\Sigma}$ is a topological group if and only if $S_{\Sigma}$ is a normal subgroup of $G$. In particular, $G^H$ 
is a topological group 
if and only if $H$ is a normal subgroup of $G$.

Using the above subgroup topology and considering subgroups of the homotopy groups mentioned in Section 2, we can introduce some topologies 
on the homotopy groups $\pi_n(X,x_0)$ such as the $n$-small subgroup topology, $\pi_n^{\pi^{s}(X,x_0)}(X,x_0)$, 
the $n$-small generated subgroup topology, $\pi_n^{\pi^{sg}(X,x_0)}(X,x_0)$, the $n$-path Spanier subgroup topology, 
$\pi_n^{\widetilde{\pi}^{sp}(X,x_0)}(X,x_0)$, the $n$-Spanier subgroup topology, $\pi_n^{\pi^{sp}(X,x_0)}(X,x_0)$, 
the $n$-thick Spanier subgroup topology, $\pi_n^{\pi^{tsp}(X,x_0)}(X,x_0)$, the $n$-weak thick Spanier subgroup topology, 
$\pi_n^{\pi^{wtsp}(X,x_0)}(X,x_0)$.

In the following, we review various well-known topologies that have been defined on the homotopy group $\pi_n(X,x_0)$.

1. \textbf{The whisker topology}:  Define 
\[\Sigma^{wh} = \{ \pi_n(i)\pi_n(U,x_0) |\ U \text{ is an open subset of $X$ containing } x_0 \}.\] 
Then it is easy to see that $\Sigma$ is a neighbourhood family on $\pi_n(X,x_0)$. 
The \emph{whisker topology} on the nth homotopy group, $\pi_n(X,x_0)$, 
of a pointed topological space $(X,x_0)$ is the subgroup topology determined by the neighbourhood family $\Sigma^{wh}$ 
which is denoted by $\pi_n^{wh}(X,x_0)$ 
and by \cite[Proposition 2.6]{BM20} $\pi_n^{wh}(X,x_0)$ is a topological group for $n \geq 2$ (see \cite{BM20}). \label{1.1}

2. \textbf{The compact-open quotient topology}:  Let $(X,x_0)$ be a pointed topological space and $\Omega^n(X,x_0)$ 
denote the space of n-loops in X based at $ x_0$. 
There exists the usual compact-open topology on $\Omega^n(X,x_0)$ which is generated by subbasis sets 
$\langle K,U\rangle=\{\alpha\mid \alpha(K)\subseteq U\}$ 
for compact $K\subseteq I^n$ and open $U\subseteq X$. By considering the surjection map 
$q_n:\Omega^n(X,x_0)\to\pi_n(X,x_0),\ q(\alpha)=[\alpha]$ one can 
equip $\pi_n(X,x_0)$  with the quotient topology with respect to the map $q_n:\Omega^n(X,x_0)\to\pi_n(X,x_0)$ 
which is denoted by $\pi_n^{qtop}(X,x_0)$ 
(see \cite{G8, G10}).\label{1.2}

3. \textbf{The tau topology}: In \cite{B13}, it was observed that for any group with topology $G$, there is a finest group topology on the group $G$, 
which is coarser than that of $G$. The resulting topological group is denoted $\tau (G)$. It is often convenient to think of $\tau$ as a 
functor which removes the smallest number of open sets from the topology of $G$ so that a topological group is obtained. 
The quasitopological group $\pi_n^{qtop}(X,x_0)$ with this topology is denoted by $\pi_n^{\tau}(X,x_0)$. \label{1.3}

4. \textbf{The lim topology}: Let $U$ be an open neighborhood of $x_0$ and 
$\widehat{U} = \{ [f]\ |\ f \in \Omega^n(X,x_0) \text{ and } Im f \subseteq U \}$. 
Since $\pi_n(X,x_0)$ is an abelian group, for $n \geq 2$, the filter base $\{\widehat{U}\}$ forms a fundamental system of neighborhoods 
of the identity element 
$[c_{x_0}]$, and hence $\pi_n(X,x_0)$ with this topology becomes a topological group, denoted by $\pi_n^{lim}(X,x_0)$ (see \cite{G10}).

5. \textbf{The Spanier topology}: Bahredar et al. in \cite[Proposition 4.2]{BKP24} showed that 
$\Sigma^{Span} = \{\pi^{sp}(\mathcal{U},x_0) |\  \mathcal{U} \text{ is an open cover of $X$} \}$ is a neighbourhood family on $\pi_n(X,x_0)$ 
and they called the subgroup topology on $\pi_n(X,x_0)$ determined by $\Sigma^{Span}$ the \emph{Spanier topology} (see \cite{BKP24}).

6. \textbf{The thick Spanier topology}: Bahredar et al. in \cite[Proposition 4.2]{BKP24} showed that 
$\Sigma^{tSpan} = \{\pi^{tsp}(\mathcal{U},x_0) |\  \mathcal{U} \text{ is an open cover of $X$} \}$ is a neighbourhood family on $\pi_n(X,x_0)$ 
and they called the subgroup topology on $\pi_n(X,x_0)$ determined by $\Sigma^{tSpan}$ 
the \emph{thick Spanier topology} (see \cite{BKP24}).

7. \textbf{The weak thick Spanier topology}: Bahredar et al. in \cite[Proposition 4.2]{BKP24} showed that 
$\Sigma^{wtSpan} = \{\pi^{tsp}(\mathcal{U},x_0) |\  \mathcal{U} \text{ is an open cover of $X$} \}$ is a neighbourhood family on $\pi_n(X,x_0)$ 
and they called the subgroup topology on $\pi_n(X,x_0)$ determined by $\Sigma^{wtSpan}$ the \emph{weak thick Spanier topology} 
(see \cite{BKP24}).

8. \textbf{The path Spanier topology}: It is easy to see that 
\[\Sigma^{pSpan} = \{\widetilde{\pi}^{sp}(\mathcal{V},x_0) |\  \mathcal{V} \text{ is a path open cover of $X$ } \}\] 
is a neighbourhood family on $\pi_n(X,x_0)$. Therefore, we call the subgroup topology on $\pi_n(X,x_0)$ determined by 
$\Sigma^{pSpan}$ the \emph{path Spanier topology}.

9. \textbf{The shape topology}: Let $cov(X)$ be the directed set of pairs $(\mathcal{U},U_0)$ where $\mathcal{U}$ 
is a locally finite open cover of $X$ and $U_0$ is a distinguished element of $\mathcal{U}$ containing $x_0$. 
Here, $cov(X)$ is directed by refinement. Given $(\mathcal{U},U_0) \in cov(X)$ let $N(\mathcal{U})$ be the abstract simplicial complex 
which is the nerve of $\mathcal{U}$. In particular, $\mathcal{U}$ is the vertex set of $U$ and the n vertices $U_1, . . . , U_n$ 
span an n-simplex if and only if $ \cap_{i=1}^{n}U_i \neq \varnothing$. The geometric realization $|N(\mathcal{U})|$
 is a polyhedron and thus $\pi_n(|N(\mathcal{U})|,U_0)$ may be regarded naturally as a discrete group, e.g. if it is given the quotient topology. 
Given a pair $(\mathcal{V},V_0)$ which refines 
$(\mathcal{U},U_0)$, a simplicial map $p_{\mathcal{U}\mathcal{V}} : |N(\mathcal{V})| \to |N(\mathcal{U})|$ 
is constructed by sending a vertex $V \in \mathcal{V}$ to some $U \in \mathcal{U}$ for which $V \subseteq U$ 
(in particular, $V_0$ is mapped to $U_0$) and extending linearly. The map $p_{\mathcal{U}\mathcal{V}}$ 
is unique up to homotopy and thus induces a unique homomorphism 
$p_{\mathcal{U}\mathcal{V}_{\sharp}} : \pi_n(|N(\mathcal{V})|,V_0) \to \pi_n(|N(\mathcal{U})|,U_0)$. 
The inverse system $(\pi_n(|N(\mathcal{U})|,U_0), p_{\mathcal{U}\mathcal{V}_{\sharp}}, cov(X))$ 
of discrete groups is the nth pro-homotopy group and the limit $\check{\pi}_n(X,x_0)$ 
(topologized with the usual inverse limit topology) is the \textit{n-th shape homotopy group}.
The induced continuous homomorphism $p_{\mathcal{U}_{\sharp}} : \pi_n(X, x_0) \to \pi_n(|N(\mathcal{U})|,U_0)$ 
satisfies $p_{\mathcal{U}_{\sharp}} \circ p_{\mathcal{U}\mathcal{V}_{\sharp}} = p_{\mathcal{V}_{\sharp}}$ 
whenever $(\mathcal{V}, V_0)$ refines $(\mathcal{U},U_0)$. Thus there is a canonical, continuous homomorphism 
$\Psi_n : \pi_n(X,x_0) \to \check{\pi}_n(X,x_0)$ to the n-th shape homotopy group, given by 
$\Psi_n([\alpha]) = ([p_{\mathcal{U}} \circ \alpha])_{\mathcal{U}}$.\\
The shape topology on $\pi_n(X, x_0)$ is the initial topology with respect to the first shape homomorphism 
$\Psi_n: \pi_n(X, x_0)\rightarrow \check{\pi}_n(X,x_0)$. Let $\pi_n^{sh}(X, x_0)$ denote $\pi_n(X, x_0)$ 
equipped with the shape topology which is a topological group 
(see \cite[Definition 3.1]{BF23}).

10. \textbf{The pseudometric topology}: Let $(X,d)$ be a path-connected metric space and consider the uniform metric 
$\mu (\alpha , \beta) = \underset{t \in [0,1]^n}{sup} \{d(\alpha(t),\beta(t))\}$ on $\Omega^n(X,x_0)$. Then the function 
$\rho : \pi_n(X,x_0) \times \pi_n(X,x_0) \to [0,\infty)$, $\rho (a,b) = inf{\mu (\alpha,\beta) | \alpha \in a,\beta \in b}$, is a pseudometric 
on $\pi_n(X,x_0)$ \cite[Theorem 4.5]{BF23}. Let $\pi_n^{met}(X,x_0)$ denote the n-th homotopy group equipped with the topology induced by the 
pseudometric $\rho$. Brazas and Fabel called this topology the pseudometric topology (induced by d) (see \cite{BF23}).

Some researchers have attempted to compare the above topologies as follows.
Ghane et al. in \cite[p. 263]{G10} proved that $\pi_n^{qtop}(X,x_0)$ is coarser than $\pi_n^{lim}(X,x_0)$. Babaee et al. in \cite[p. 1441]{BM20} 
mentioned that $\pi_n^{lim}(X,x_0)$ coincides with the whisker topology $\pi_n^{wh}(X,x_0)$. 
It is shown in \cite[Proposition 3.2]{BF23} that the shape topology of $\pi_n^{sh}(X, x_0)$ is coarser than that of $\pi_n^{\tau}(X, x_0)$.
Note that by \cite{BKP24} one can show that $\pi_n^{Span}(X, x_0)$ is finer than $\pi_n^{tSpan}(X, x_0)$.
By considering the definitions of $\pi_n^{qtop}(X,x_0)$ and $\pi_n^{\tau}(X, x_0)$ it is easy to see that 
$\pi_n^{qtop}(X,x_0)$ is finer than $\pi_n^{\tau}(X, x_0)$.
The Spanier topology is always finer than the shape topology on $\pi_n(X,x_0)$ \cite[Remark 5.5]{AB23}.
Brazas and Fabel proved that if $X$ is a compact metric space, then the topology of $\pi_n^{met}(X,x_0)$ is at least as fine as that of 
$\pi_n^{sh}(X,x_0)$ \cite[Proposition 5.2]{BF23}.\\

\section{Comparison of Topologies on Homotopy Groups}
In this section, using some properties of the subgroup topology presented in \cite{BS, SM} and some results of other researchers, 
we intend to compare the various topologies on homotopy groups mentioned in Section 3 with each other as much as possible.

In the following theorem, we compare the Spanier topology on homotopy groups with the subgroup topology induced by the $n$-Spanier subgroup.
\begin{theorem}\label{4.1}
$(i)$ If $H=\pi^{sp}(X,x_0)$, then $\pi_n^{Span}(X,x_0) \preccurlyeq \pi_n^H(X,x_0)$.\\
$(ii)$ If $\pi_n^{Span}(X,x_0)=\pi_n^K(X,x_0)$ for a subgroup $K$ of $\pi_n(X,x_0)$, then $K=\pi^{sp}(X,x_0)$.
\end{theorem}
\begin{proof}
$(i)$ It is known that $\pi_n^{Span}(X,x_0)$ has subgroup topology with respect to 
$\Sigma^{Span}=\{K\leqslant \pi_n(X,x_0) \ \mid \ K \text{ is an n-Spanier subgroup} \}$ (see \cite{BKP24}).
Since $H=\pi^{sp}(X,x_0)$,  $\Sigma^H=\{K\leqslant \pi_n(X,x_0) \ \mid \pi^{sp}(X,x_0)\subseteq K \}$
and $\pi^{sp}(X,x_0)=\cap_{K\in \Sigma^{Span}} K$, we have
$\Sigma^{Span}\subseteq \Sigma^H$. Hence $\pi_n^{Span}(X,x_0)\preccurlyeq \pi_n^H(X,x_0)$.\\
$(ii)$ It holds by \cite[Theorem 2.3]{SM}.
\end{proof}

\begin{remark}\label{4.2}
Bahredar et al. in \cite{BKP24} mentioned that if $\mathcal{U}$ is an open cover of $X$, 
then $\pi^{sp}(\mathcal{U},x_0) \subseteq \pi^{tsp}(\mathcal{U},x_0) \subseteq \pi^{wtsp}(\mathcal{U},x_0)$. 
Also, it is easy to see that if $\mathcal{V}$ is a path open cover of $X$, 
then $\widetilde{\pi}^{sp}(\mathcal{V},x_0) \subseteq \pi^{sp}(\mathcal{V},x_0)$. 
Thus similar to the proof of the above theorem we have the following comparison of some topologies on the homotopy groups.\\ 
$(i)$ If $H=\pi^{tsp}(X,x_0)$, then $\pi_n^{tSpan}(X,x_0) \preccurlyeq \pi_n^H(X,x_0)$.\\
$(ii)$ If $\pi_n^{tSpan}(X,x_0)=\pi_n^K(X,x_0)$ for a subgroup $K$ of $\pi_n(X,x_0)$, then $K=\pi^{tsp}(X,x_0)$.\\
$(iii)$ If $H=\pi^{wtsp}(X,x_0)$, then $\pi_n^{wtSpan}(X,x_0) \preccurlyeq \pi_n^H(X,x_0)$.\\
$(iv)$ If $\pi_n^{wtSpan}(X,x_0)=\pi_n^K(X,x_0)$ for a subgroup $K$ of $\pi_n(X,x_0)$, then $K=\pi^{wtsp}(X,x_0)$.\\
$(v)$ If $H=\widetilde{\pi}^{sp}(X,x_0)$, then $\pi_n^{pSpan}(X,x_0) \preccurlyeq \pi_n^H(X,x_0)$.\\
$(vi)$ If $\pi_n^{pSpan}(X,x_0)=\pi_n^K(X,x_0)$ for a subgroup $K$ of $\pi_n(X,x_0)$, then $K=\widetilde{\pi}^{sp}(X,x_0)$.\\
$(vii)$ $\pi_n^{wtSpan}(X,x_0) \preccurlyeq \pi_n^{tSpan}(X,x_0) \preccurlyeq \pi_n^{Span}(X,x_0) \preccurlyeq \pi_n^{pSpan}(X,x_0)$.\\

\end{remark}

\begin{remark}\label{4.3}
$(i)$ In \cite[Proposition 3.15]{BKP24}, Bahredar et al. proved that for any pointed space $(X,x_0)$, $\pi^{wtsp}(X,x_0) \leq ker\Psi_X$, 
where $\Psi_X$ is the canonical map from the nth homotopy group to the nth shape homotopy group. 
It is clear that $ker\Psi_X \leq ker p_{\mathcal{U}_{\sharp}}$ On the other hand, Brazas and Fabel in \cite{BF23} 
mentioned that the shape topology on homotopy groups is generated by left cosets of subgroups of the form $ker p_{\mathcal{U}_{\ast}}$. 
Thus, we have the following result, which refines Corollary 5.7 from \cite{BF23}.

 For any pointed space $(X,x_0)$ we have
\[\pi_n^{sh}(X,x_0) \preccurlyeq \pi_n^{wtSpan}(X,x_0) \preccurlyeq \pi_n^{tSpan}(X,x_0) \preccurlyeq \pi_n^{Span}(X,x_0) 
\preccurlyeq \pi_n^{pSpan}(X,x_0).\]
$(ii)$ If $X$ is $T_1$ and paracompact space, then $\pi^{sp}(X,x_0) = \pi^{tsp}(X,x_0) = \pi^{wtsp}(X,x_0)$. 
Thus $\pi_n^{\pi^{wtsp}(X,x_0)}(X,x_0) = \pi_n^{\pi^{tsp}(X,x_0)}(X,x_0) = \pi_n^{\pi^{sp}(X,x_0)}(X,x_0)$. 
Moreover, this identity holds if $X$ is metrizable or path connected topological group (see \cite[Theorem 3.11, Theorem 3.12]{BKP24}).
\end{remark}

\begin{definition}\label{4.4}
Let $X$ be a topological space and $H$ be a subgroup of $\pi_n(X,x_0)$. 
Then $X$ is called \textit{n-semilocally $H$-connected at $x_0$} if there exists an open neighborhood 
$U$ of $x_0$ with $\pi_n(i)(\pi_n(U,x_0)) \leq H$.
\end{definition}

Abdullahi et al. in \cite[Theorem 4.2]{A16} proved that if $H \leq \pi_1(X,x_0)$, 
then $X$ is semilocally $H$-connented at $x_0$ if and only if $H$ is open in $\pi_1^{wh}(X,x_0)$. 
Now we generalize this result to homotopy groups as follows.

\begin{theorem}\label{4.5}
Let $H \leq \pi_n(X,x_0)$, then $X$ is n-semilocally $H$-connected at $x_0$ if and only if $H$ is an open subgroup of $\pi_n^{wh}(X,x_0)$.
\end{theorem}

\begin{proof}
Let $X$ be n-semilocally $H$-connected at $x_0$, then there is an open neighborhood $U$ of $x_0$ 
such that $\pi_n(i)(\pi_n(U,x_0)) \leq H$. By definition we have $\pi_n(i)(\pi_n(U,x_0)) \in \Sigma^{wh}$, therefore $H$ is open in $\pi_n^{wh}(X,x_0)$.
Conversely, if $H$ is an open subgroup of $\pi_n^{wh}(X,x_0)$, then there exists an open neighborhood $U$ of $x_0$ 
such that $\pi_n(i)(\pi_n(U,x_0))$ $\leq H$. Hence $X$ is n-semilocally $H$-connected at $x_0$.
\end{proof}
Using the above theorem and the definition of the whisker topology we have the following result.
\begin{corollary}\label{4.6}
Let $H \leq \pi_n(X,x_0)$, then $X$ is n-semilocally $H$-connected at $x_0$ if and only if $\pi_n^H(X,x_0) \preccurlyeq \pi_n^{wh}(X,x_0)$.
\end{corollary}

Let $[\alpha] \in \cap \{\pi_n(i)(\pi_n(U,x_0)) |\  U \text{ is an open neighbourhood of } x_0\}$, 
then $\alpha$ has a homotopic representative in any open neighbourhood of $x_0$, that is, $\alpha$ is a small n-loop at $x_0$. 
Thus, the infinitesimal subgroup of $\pi_n^{wh}(X,x_0)$ is equal to $\pi^s(X,x_0)$. 
Hence, $\pi_n^{wh}(X,x_0) \preccurlyeq \pi_n^{\pi^s(X,x_0)}(X,x_0)$. 
Note that the equality does not hold in general. By considering the n-dimensional Hawaiian earring , 
$\mathbb{HE}^n$, at the origin $\theta$, we have $\pi_n^{\pi^s(\mathbb{HE}^n,\theta)}(\mathbb{HE}^n,\theta)$ 
is discrete but $\pi_n^{wh}(\mathbb{HE}^n,\theta)$ is not discrete (see \cite[Example 4.6]{BM20}). 
In the following we present an equivalent condition for the equality.

\begin{corollary}\label{4.7}
Let $(X,x_0)$ be a pointed topological space, then 
\[\pi_n^{wh}(X,x_0 \preccurlyeq \pi_n^{\pi^s(X,x_0)}(X,x_0).\]
Also, we have $X$ is n-semilocally $\pi^s(X,x_0)$-connected at $x_0$ if and only if
$\pi_n^{wh}(X,x_0)= \pi_n^{\pi^s(X,x_0)}(X,x_0)$.
Moreover, if $\pi_n^{wh}(X,x_0)=\pi_n^K(X,x_0)$ for a subgroup $K$ of $\pi_n(X,x_0)$, then $K=\pi^{s}(X,x_0)$.
\end{corollary}
\begin{proof}
Since the infinitesimal subgroup of $\pi_n^{wh}(X,x_0)$ is equal to $\pi^s(X,x_0)$ we have
 $\pi_n^{wh}(X,x_0) \preccurlyeq \pi_n^{\pi^s(X,x_0)}(X,x_0)$.
Let $X$ be n-semilocally $\pi^s(X,x_0)$-connected at $x_0$. 
Then by Corollary \ref{4.6}, $\pi_n^{\pi^s(X,x_0)}(X,x_0) \preccurlyeq \pi_n^{wh}(X,x_0)$. 
Thus we have $\pi_n^{wh}(X,x_0)= \pi_n^{\pi^s(X,x_0)}(X,x_0)$.
The converse is true by Corollary \ref{4.6}.
By \cite[Theorem 2.3]{SM}  if $\pi_n^{wh}(X,x_0)=\pi_n^K(X,x_0)$ for a subgroup $K$ of $\pi_n(X,x_0)$, then $K=\pi^{s}(X,x_0)$.
\end{proof}

\begin{remark}\label{4.8}
$(i)$ Note that by \cite[Proposition 3.2]{BF23} we have $\pi_n^{sh}(X,x_0) \preccurlyeq \pi_n^{\tau}(X,x_0)$. 
Also, by \cite[Lemma 3.10]{AB23} we have $\pi_n^{sh}(X,x_0) \preccurlyeq \pi_n^{Span}(X,x_0)$. 
Note that by \cite[Lemma 5.1]{AB23} the equality holds if $X$ is paracompact, Hausdorf and $LC^{n-1}$ space. 
A space $X$ is called $LC^n$ at $x \in X$ if for every neighborhood $U$ of $x$, 
there exists a neighborhood $V$ of $x$ such that $V \subseteq U$ and such that for all $0 \leq k \leq n$ ($k < \infty$ if $n = \infty$) 
every map $f : \partial\Delta_{k+1} \to V$ extends to a map $g : \triangle_{k+1} \to U$ (see \cite[Definition 4.3]{AB23}). \\
$(ii)$ Note that  $\pi_n^{qtop}(X,x_0)$ is coarser than $\pi_n^{lim}(X,x_0)$ (see \cite[p. 263]{G10}). 
Also, by \cite[Page 1441]{BM20} $\pi_n^{lim}(X,x_0)$ coincides with the whisker topology $\pi_n^{wh}(X,x_0)$. \\
$(iii)$ By \cite[Proposition 4.13]{BF23}, $\pi_n^{qtop}(X,x_0)$ and $\pi_n^{\tau}(X,x_0)$ are at least as fine as that of $\pi_n^{met}(X,x_0)$.
Also, by \cite[Corollary 5.12]{BF23}, $\pi_n^{met}(X,x_0) \preccurlyeq \pi_n^{Span}(X,x_0)$.\\
$(iv)$ Since $\pi_n(X,x_0)$ is an abelian group for each $n\geq 2$ by \cite[Theorem 2.2]{SM} 
 the topologized homotopy group $\pi_n^H(X,x_0)$ is a topological group for any subgroup $H$ of $\pi_n(X,x_0)$,.\\
$(v)$ In \cite[Theorem 2.4]{SM} the authors proved that if $H\leq K\leq G$, then $G^K \preccurlyeq G^H$. Thus we have
\[\pi_n^{\pi^{wtsp}(X,x_0)}(X,x_0) \preccurlyeq \pi_n^{\pi^{tsp}(X,x_0)}(X,x_0) \preccurlyeq \pi_n^{\pi^{sp}(X,x_0)}(X,x_0)\] 
\[\preccurlyeq \pi_n^{\widetilde{\pi}^{sp}(X,x_0)}(X,x_0) \preccurlyeq \pi_n^{\pi^{sg}(X,x_0)}(X,x_0) \preccurlyeq \pi_n^{\pi^s(X,x_0)}(X,x_0).\]
\end{remark}

Finally, we summarize all results of this section in order to compare various topologies on the homotopy groups in the following diagram
(note that $A \longrightarrow B$ means that $A\preccurlyeq B$).

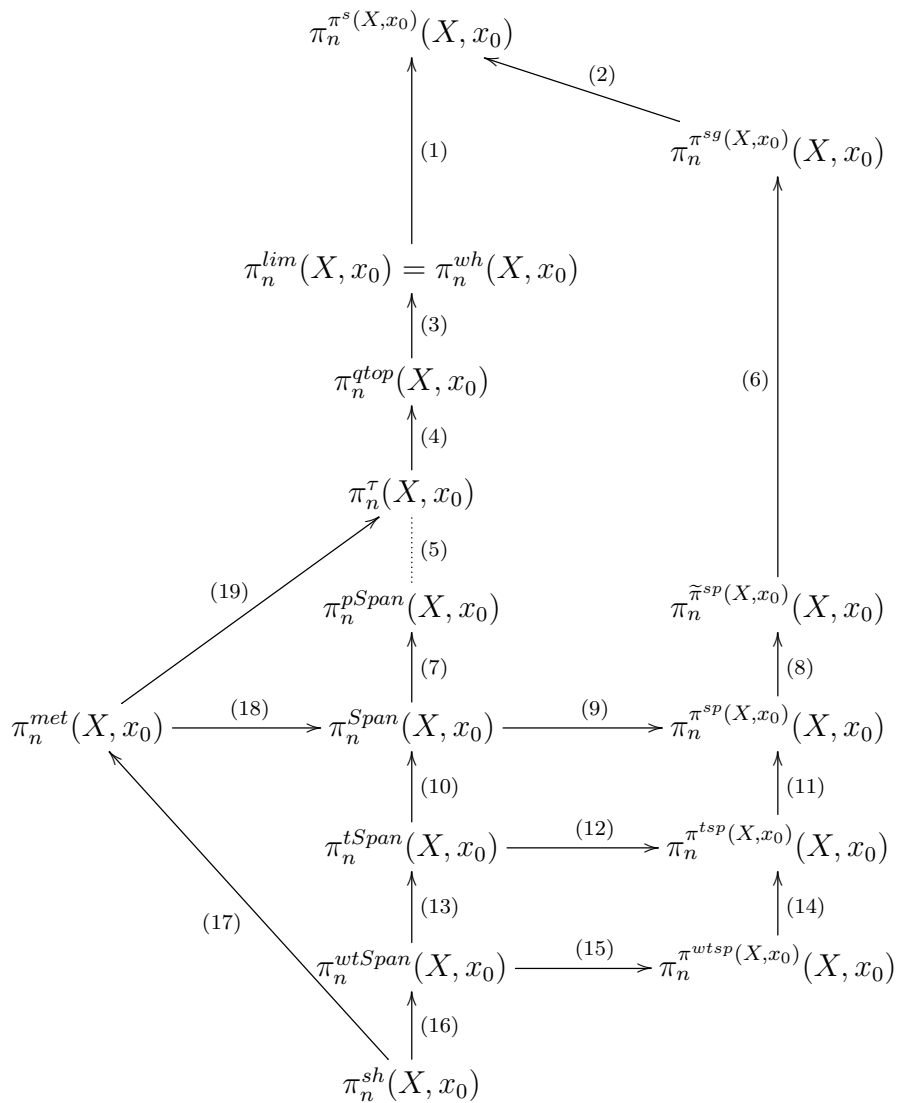
\[ \xymatrix{
& \pi_n^{\pi^s(X,x_0)}(X,x_0) & \\
& & \pi_n^{\pi^{sg}(X,x_0)}(X,x_0) \ar[ul]_{(2)} \\
& \pi_n^{lim}(X,x_0) = \pi_n^{wh}(X,x_0) \ar[uu]_{(1)} & \\
& \pi_n^{qtop}(X,x_0) \ar[u]_{(3)} & \\
& \pi_n^{\tau}(X,x_0) \ar[u]_{(4)} & \\
& \pi_n^{pSpan}(X,x_0) \ar@{.}[u]_{(5)} & \pi_n^{\widetilde{\pi}^{sp}(X,x_0)}(X,x_0) \ar[uuuu]^{(6)}\\
\pi_n^{met}(X,x_0) \ar[r]^{(18)} \ar[uur]^{(19)}  & \pi_n^{Span}(X,x_0) \ar[u]_{(7)} \ar[r]^{(9)} & \pi_n^{\pi^{sp}(X,x_0)}(X,x_0) \ar[u]_{(8)} \\
& \pi_n^{tSpan}(X,x_0) \ar[u]_{(10)} \ar[r]^{(12)} &  \pi_n^{\pi^{tsp}(X,x_0)}(X,x_0) \ar[u]_{(11)} \\
& \pi_n^{wtSpan}(X,x_0) \ar[u]_{(13)} \ar[r]^{(15)} & \pi_n^{\pi^{wtsp}(X,x_0)}(X,x_0) \ar[u]_{(14)} \\
& \pi_n^{sh}(X,x_0) \ar[u]_{(16)} \ar[uuul]^{(17)} & \\
}\]

In the following, according to the enumeration in the above diagram, we give references and complementary notes for each arrow.
\begin{itemize}
\item[(1)] See Corollary \ref{4.7}. By Theorem \ref{4.5} the equality holds if and only if $X$ is n-semilocally $\pi^s(X,x_0)$-connected at $x_0$.
Since $\pi^s(\mathbb{HE}^n,\theta) = 1$ the strict inequality holds for $\mathbb{HE}^n$ at the origin $\theta$ (see \cite[Example 4.6]{BM20}).
\item[(2)] See Remark \ref{4.8} $(v)$.
 The equality holds if and only if $\pi^s(X,x_0)= \pi^{sg}(X,x_0)$ by \cite[Theorem 2.3]{SM}.
 \item[(3)] See Remark \ref{4.8} $(ii)$.
 \item[(4)] See Definition \ref{1.3}.
 The equality holds if and only if $\pi_n^{qtop}(X,x_0)$ is a topological group.
 The strict inequality holds for space $X$ in \cite{F12}, because $\pi_n^{qtop}(X,p)$ is not a topological group.
 \item[(5)] By Remark \ref{4.8} $(i)$ if $X$ is paracompact, Hausdorf and $LC^{n-1}$ space, 
 then $\pi_n^{pSpan}(X,x_0)=\pi_n^{Span}(X,x_0)=\pi_n^{sh}(X,x_0) \preccurlyeq \pi_n^{\tau}(X,x_0)$. 
 We conjecture that $\pi_n^{\tau}(X,x_0)$ is finer than $\pi_n^{pSpan}(X,x_0)$ under mild conditions.
 \item[(6)] See Remark \ref{4.8} $(v)$.
 The equality holds if and only if $\widetilde{\pi}^{sp}(X,x_0)= \pi^{sg}(X,x_0)$ by \cite[Theorem 2.3]{SM}.
 \item[(7)] See Remark \ref{4.3} $(i)$. By \cite[Theorem 2.4]{SM} if $\widetilde{\pi}^{sp}(X,x_0)= \pi^{sp}(X,x_0)$ and 
 $\pi^{sp}(X,x_0)$ is open in $\pi^{Span}(X,x_0)$, then the equality holds.Also, by \cite[Theorem 2.3]{SM} if the equality holds, 
 then $\widetilde{\pi}^{sp}(X,x_0)= \pi^{sp}(X,x_0)$.
 \item[(8)] See Remark \ref{4.8} $(v)$.
 The equality holds if and only if $\widetilde{\pi}^{sp}(X,x_0)= \pi^{sp}(X,x_0)$ by \cite[Theorems 2.3, 2.4]{SM}.
 \item[(9)] See Theorem \ref{4.1}. By \cite[Theorem 2.4]{SM} if $\pi^{sp}(X,x_0)$ is open in $\pi^{Span}(X,x_0)$, then the equality holds.
 \item[(10)] See Remark \ref{4.3} $(i)$. By \cite[Theorem 2.4]{SM} if $\pi^{sp}(X,x_0)= \pi^{tsp}(X,x_0)$ and 
 $\pi^{tsp}(X,x_0)$ is open in $\pi^{tSpan}(X,x_0)$, then the equality holds. Also, by \cite[Theorem 2.3]{SM} if the equality holds, 
 then $\pi^{sp}(X,x_0)= \pi^{tsp}(X,x_0)$.
 \item[(11)] See Remark \ref{4.8} $(v)$.
  The equality holds if and only if $\pi^{sp}(X,x_0)= \pi^{tsp}(X,x_0)$ by \cite[Theorems 2.3, 2.4]{SM}.
 \item[(12)] See Remark \ref{4.2} $(i)$. By \cite[Theorem 2.4]{SM} if  $\pi^{tsp}(X,x_0)$ is open in $\pi^{tSpan}(X,x_0)$, then the equality holds.
 \item[(13)] See Remark \ref{4.3} $(i)$. By \cite[Theorem 2.4]{SM} if $\pi^{wtsp}(X,x_0)= \pi^{tsp}(X,x_0)$ and 
 $\pi^{wtsp}(X,x_0)$ is open in $\pi^{wtSpan}(X,x_0)$, then the equality holds. Also, by \cite[Theorem 2.3]{SM} if the equality holds,
  then $\pi^{wtsp}(X,x_0)= \pi^{tsp}(X,x_0)$.
 \item[(14)] See Remark \ref{4.8} $(v)$.
  The equality holds if and only if $\pi^{tsp}(X,x_0)= \pi^{wtsp}(X,x_0)$ by \cite[Theorem 2.3, Theorem 2.4]{SM}.
 \item[(15)] See Remark \ref{4.2} $(iii)$. By \cite[Theorem 2.4]{SM} if $\pi^{wtsp}(X,x_0)$ is open in\\ $\pi^{wtSpan}(X,x_0)$, then the equality holds.
 \item[(16)] See Remark \ref{4.3} $(i)$.
 \item[(17)] It holds for compact metric spaces by \cite[Proposition 5.2]{BF23}.
 If $X$ is a path-connected compact metric space, then the equality holds in the following two case:\\
 $(i)$ if $X$ is $LC^{n-1}$ \\
 $(ii)$ if $X = \underleftarrow{lim}_{j \in \mathbb{N}}(X_j,r_{j+1,j})$ is an inverse limit of finite polyhedra where the bonding maps 
 $r_{j+1,j} : X_{j+1} \to X_j$ are retractions (see \cite[Theorem 1.1]{BF23}).
 \item[(18)] See Remark \ref{4.8} $(iii)$.
 \item[(19)] See Remark \ref{4.8} $(iii)$.
\end{itemize}

In order to investigate further the above diagram, we raise some questions in the following which we are interested in finding answers to them 
(the questions are numbered to correspond with the respective arrow numbers in the diagram.).
\begin{itemize}
\item[(Q2)] Is there a space $X$ for which the strict inequality $\pi_n^{\pi^{sg}(X,x_0)}(X,x_0) \prec $\\ $ \pi_n^{\pi^s(X,x_0)}(X,x_0)$ holds when $n\geq2$?
Note that the strict inequality holds when $n=1$ for $\mathbb{HA}$ at $b\neq 0$ since $\pi^{s}(\mathbb{HA},b)=1$ 
and $\pi^{sg}(\mathbb{HA},b)=\pi_1(\mathbb{HA},b)$ (\cite [Example 3.12]{A16}).
\item[(Q3)] Is there a necessary and sufficient condition on $X$ for the equality $\pi_n^{qtop}(X,x_0)= \pi_n^{wh}(X,x_0)$ when $n\geq2$?
Note that if $X$  is locally path connected, then the equality holds if and only if $X$ is SLT at $x_0$ (see \cite[Corollary 3.3]{ PA17}). 
Also, is there a space $X$ for which the strict inequality $\pi_n^{qtop}(X,x_0) \prec \pi_n^{wh}(X,x_0)$ holds when $n\geq2$?
Note that the strict inequality holds when $n=1$ for $\mathbb{HE}$ (see \cite[Example 3.25]{A20}).
\item[(Q5)] Is it true that $\pi_n^{pSpan}(X,x_0) \preccurlyeq \pi_n^{\tau}(X,x_0)$?

Nasri et. al in \cite[Theorem 2.1]{N20} proved that $\pi_n^{qtop}(X,x_0) \cong \pi_{n-k}^{qtop}(\Omega^k(X,x_0),e_{x_0})$, 
 where $e_{x_0}$ is a constant $k$-loop in $X$ at $x_0$. Also they presented the following commutative diagram:
 \[ \xymatrix{ 
 \Omega^n(X,x_0) \ar[r]^{\phi} \ar[d]^{q} & \Omega^{n-k}(\Omega^k(X,x_0),e_{x_0}) \ar[d]^{q} \\
 \pi_n^{qtop}(X,x_0) \ar[r]^{\phi_{\ast}}  & \pi_{n-k}^{qtop}(\Omega^k(X,x_0),e_{x_0})} \]
where $\phi : \Omega^n(X,x_0) \to \Omega^{n-k}(\Omega^k(X,x_0),e_{x_0})$ given by $\phi (f) = f^{\sharp}$ 
is a homeomorphism with inverse $g \mapsto g^{\flat}$ in the sense of \cite{Rot}. 
Since $q$ is a quotient map, the homomorphism $\phi_{\ast}$ is an isomorphism between quasitopological homotopy groups.
By the above diagram if we can prove that $\phi_{\ast}(\widetilde{\pi}_n^{sp}(\mathcal{V},x_0))=\widetilde{\pi}_{n-1}^{sp}(\mathcal{V'},e_{x_0}) $, 
then we can conclude that $\pi_n^{pSpan}(X,x_0) \preccurlyeq \pi_n^{\tau}(X,x_0)$.
\item[(Q6)] Is there a space $X$ for which the strict inequality $\pi_n^{\widetilde{\pi}^{sp}(X,x_0)}(X,x_0) \prec$\\ 
$\pi_n^{\pi^{sg}(X,x_0)}(X,x_0)$ holds when $n\geq 2$? Note that the strict inequality holds 
when $n=1$ for the space $RX$ in \cite[Example 2.5]{A16} (see \cite[Example 2]{SM}).
\item[(Q7)] Is there a space $X$ for which the strict inequality $\pi_n^{Span}(X,x_0) \prec \pi_n^{pSpan}(X,x_0)$ holds?
\item[(Q8)] Is there a space $X$ for which the strict inequality $\pi_n^{\pi^{sp}(X,x_0)}(X,x_0) \prec$\\ $ \pi_n^{\widetilde{\pi}^{sp}(X,x_0)}(X,x_0)$ holds?
\item[(Q9)] Is there a space $X$ for which the strict inequality $\pi_n^{Span}(X,x_0) \prec \pi_n^{\pi^{sp}(X,x_0)}(X,x_0)$ holds?
\item[(Q10)] Is there a space $X$ for which the strict inequality $\pi_n^{tSpan}(X,x_0) \prec \pi_n^{Span}(X,x_0)$ holds?
\item[(Q11)] Is there a space $X$ for which the strict inequality $\pi_n^{\pi^{tsp}(X,x_0)}(X,x_0) \prec$\\ $ \pi_n^{\pi^{sp}(X,x_0)}(X,x_0)$ holds?
\item[(Q12)] Is there a space $X$ for which the strict inequality $\pi_n^{tSpan}(X,x_0) \prec$\\ $ \pi_n^{\pi^{tsp}(X,x_0)}(X,x_0)$ holds?
\item[(Q13)] Is there a space $X$ for which the strict inequality $\pi_n^{wtSpan}(X,x_0) \prec \pi_n^{tSpan}(X,x_0)$ holds?
\item[(Q14)] Is there a space $X$ for which the strict inequality $\pi_n^{\pi^{wtsp}(X,x_0)}(X,x_0) \prec$\\ $ \pi_n^{\pi^{tsp}(X,x_0)}(X,x_0)$ holds?
\item[(Q15)] Is there a space $X$ for which the strict inequality $\pi_n^{wtSpan}(X,x_0) \prec$\\ $ \pi_n^{\pi^{wtsp}(X,x_0)}(X,x_0)$ holds?
\item[(Q16)] Is there a necessary and sufficient condition on $X$ for the equality $\pi_n^{sh}(X,x_0) = \pi_n^{wtSpan}(X,x_0)$?
Is there a space $X$ for which the strict inequality $\pi_n^{sh}(X,x_0) \prec \pi_n^{wtSpan}(X,x_0)$ holds?
\item[(Q17)] Is there a compact metric space $X$ for which the strict inequality $\pi_n^{sh}(X,x_0) \prec \pi_n^{met}(X,x_0)$ holds?
\item[(Q18)] Is there a necessary and sufficient condition on $X$ for the equality $\pi_n^{met}(X,x_0) = \pi_n^{Span}(X,x_0)$?
Is there a space $X$ for which the strict inequality $\pi_n^{met}(X,x_0) \prec \pi_n^{Span}(X,x_0)$ holds?
\item[(Q19)] Is there a necessary and sufficient condition on $X$ for the equality $\pi_n^{met}(X,x_0) = \pi_n^{\tau}(X,x_0)$?
Is there a space $X$ for which the strict inequality $\pi_n^{met}(X,x_0) \prec \pi_n^{\tau}(X,x_0)$ holds?
\end{itemize}

\textbf{Declaration of Interest Statement }\\
The authors declare that they have no known competing financial interests or personal relationships that could have appeared to influence the work reported in this paper.


\begin{thebibliography}{20}

\bibitem{A16}
M. Abdullahi Rashid, B. Mashayekhy, H. Torabi, S.Z. Pashaei,
\newblock{On subgroups of topologized fundamental groups and generalized coverings},
\newblock{\em Bull. Iranian Math. Soc.},
\textbf{43}(7), 2349-2370, (2017).

\bibitem{A20}
M. Abdullahi Rashid, N. Jamali, B. Mashayekhy, S.Z. Pashaei, H. Torabi,  
\newblock{On subgroup topologies on fundamental groups}, 
\newblock{\em Hacet. J. Math. Stat.},
\textbf{49}(3), 935-949, (2020).

\bibitem{AB23}
J. Aceti, J. Brazas,
\newblock{Elements of higher homotopy groups undetectable by polyhedral approximation},
\newblock{\em Pacific J. Math.},
\textbf{322}(2), 221-242, (2023).

\bibitem{BM20}
A. Babaee, B. Mashayekhy, H. Mirebrahimi, H. Torabi, M.A. Rashid, S.Z. Pashaei,
\newblock{On topological homotopy groups and relation to Hawaiian groups},
\newblock{\em Hacet. J. Math. Stat.},
\textbf{49}(4), 1437-1449, (2020).

\bibitem{BKP21}
A.A. Bahredar, N. Kouhestani, H. Passandideh,
\newblock The n-dimensional Spanier group,
\newblock{\em Filomat},
\textbf{35}(9), 3169-3182,(2021).

\bibitem{BKP24}
A.A. Bahredar, N. Kouhestani, H. Passandideh,
\newblock On (Thick, Weak Thick)Spanier topology on nth homotopy group,
\newblock{\em Filomat},
\textbf{38}(12), 4381–4394,(2024).

\bibitem{B13}
J. Brazas,
\newblock{The fundamental group as topological group},
\newblock{\em Topology Appl.},
\textbf{160} 170-188 (2013).

\bibitem{BS}
W.A. Bogley, A.J. Sieradski, 
\newblock Universal path spaces,
\newblock http://people.oregonstate.edu/˜bogleyw/research/ups.pdf

\bibitem{BF23}
J. Brazas, P. Fabel,
\newblock{A natural pseudometric on homotopy groups of metric spaces},
\newblock{\em Glasgow Math. J.},
\textbf{66}, 162-174, (2024).

\bibitem{CM}
J.S. Calcut, J.D. McCarthy, 
\newblock{Discreteness and homogeneity of the topological fundamental group},
\newblock{\em Topology Proc.},
\textbf{34},  339--349, (2009).

\bibitem{F9}
P. Fabel,   
\newblock{Multiplication is discontinuous in the Hawaiian earring group (with the quotient topology)}, 
\newblock {\em Bull. Pol. Acad. Sci. Math.}, 
\textbf{59} (1), 77--83, (2011).

\bibitem{F12}
P. Fabel,
\newblock{Compactly generated quasitopological homotopy groups with discontinuous multiplication},
\newblock {\em Topology Proc.},
\textbf{40},  303--309, (2012).

\bibitem{G8}
F.H. Ghane, Z. Hamed, B. Mashayekhy and H. Mirebrahimi,
\newblock{Topological homotopy groups},
\newblock{\em Bull. Belg. Math. Soc. Simon Stevin},
\textbf{15}, 455-464, (2008).

\bibitem{G10}
F.H. Ghane, Z. Hamed, B. Mashayekhy and H. Mirebrahimi,
\newblock{On topological homotopy groups of n-Hawaiian like spaces},
\newblock{\em Topology Proc.},
\textbf{36}, 255-266, (2010).

\bibitem{Naber}
G .Naber,
\newblock{Topology, Geometry and Gauge fields, Foundations},
\newblock{\em Addison-Wesley publishing company, U.S.A, INC, 1961.}

\bibitem{NM12}
T. Nasri, B. Mashayekhy, H. Mirebrahimi,
\newblock{On quasitopological homotopy groups of inverse limit spaces},
\newblock{\em Topology Proc.},
\textbf{46},  145-157, (2015).

\bibitem{N20}
T. Nasri, H. Mirebrahimi, H. Torabi,
\newblock{Some results in quasitopological homotopy groups},
\newblock {\em Ukrainian Math. J.},
\textbf{72}(12), 1663-1668, (2020).

\bibitem{PA17}
S.Z. Pashaei, B. Mashayekhy, H. Torabi, M. Abdullahi Rashid,
\newblock{Small loop transfer spaces with respect to subgroups of fundamental groups},
\newblock{Topology Appl.},
\textbf{232}, 242-255, (2017).

\bibitem{P11}
H. Passandideh, F.H. Ghane and Z. Hamed,
\newblock{On the homotopy groups of separable metric spaces},
\newblock{ \em Topology Appl.},
\textbf{158}, 1607-1614, (2011).

\bibitem{Rot}
J.J. Rotman, 
\newblock An Introduction to Algebraic Topology
\newblock  {\em Springer-Verlag, GTM 119, New York, 1988.}

\bibitem{SM}
N. Shahami, B. Mashayekhy,
\newblock{Comparison of topologies on fundamental groups with subgroup topology viewpoint},
\newblock {\em Math. Slovaca},
\textbf{75}(1), 189-204, (2025).

\bibitem{TMP}
H. Torabi, A. Pakdaman,  B. Mashayekhy,
\newblock{On the Spanier groups and covering and semicovering spaces},
\newblock {\em arXiv:1207.4394v1.}

\bibitem{W}
J. Wilkins,
\newblock{The revised and uniform fundamental groups and universal covers of geodesic spaces},
\newblock {\em Topology Appl.},
\textbf{160}, 812--835, (2013).


\end{thebibliography}
\end{document}